\documentclass[11pt]{article}

\usepackage{geometry}             
\usepackage{graphicx}
\usepackage{amsmath}
\usepackage{amssymb}
\usepackage{tikz}

\usepackage[applemac]{inputenc}

\usepackage{caption}
\usepackage{amsmath,amsfonts}
\usepackage{amsthm}
\usepackage{amssymb}
\usepackage[italian,french,english]{babel}
\usepackage{amsfonts}
\usepackage[T1]{fontenc}
\usepackage{enumerate}
\usepackage{verbatim}
\usepackage{graphicx}
\usepackage{verbatim}
\usepackage{faktor}
\usepackage{textcomp}

\usetikzlibrary{matrix,arrows}
\usepackage{epstopdf}

\newtheorem{teo}{Theorem}[section]
\newtheorem{lem}[teo]{Lemma}
\newtheorem{defi}[teo]{Definition}
\newtheorem{cor}[teo]{Corollary}
\newtheorem{prop}[teo]{Proposition}
\newtheorem*{teo*}{Theorem}
\newtheorem*{defi*}{Definition}
\theoremstyle{remark}
\newtheorem*{rem}{Remark}

\newcommand{\R}{\mathbb{R}}

\newcommand{\Q}{\mathcal{Q}}
\newcommand{\Z}{\mathbb{Z}}
\newcommand{\N}{\mathbb{N}}

\newcommand{\Hom}{\text{Hom}}

\newcommand{\s}{\sigma}
\renewcommand{\t}{\tau}

\renewcommand{\O}{\mathcal O}
\renewcommand{\P}{\mathcal P}
\newcommand{\MD}{\mathcal{MD}}
\newcommand{\Li}{\mathcal L}

\newcommand{\Spec}{\text{Spec}}

\newcommand{\G}{\Gamma}
\newcommand{\Om}{\Omega}

\newcommand{\T}{\mathbb{T}}
\renewcommand{\L}{\mathbb{L}}
\newcommand{\Sym}{\mathrm{Sym}}
\renewcommand{\Pi}{\tilde{P}_{n}}
\newcommand{\Omn}{\Omega_n}
\newcommand{\Lie}{\mathrm{Lie}}
\newcommand{\Comm}{\mathrm{Comm}}
\renewcommand{\Hom}{\mathrm{Hom}}

\renewcommand{\labelitemi}{$\bullet$}

\title{Poisson bivectors and Poisson brackets\\ on affine derived stacks}   
\date{} 
\author{Valerio Melani \\
Institut de Math\'ematiques de Jussieu - Paris Rive Gauche\\
valerio.melani@imj-prg.fr}

\begin{document}
\maketitle

\begin{abstract}
Let $A$ be a commutative dg algebra concentrated in degrees $(-\infty, m]$, and let $\Spec\, A$ be the associated derived stack. We give two proofs of the existence of a canonical map from the moduli space of shifted Poisson structures (in the sense of \cite{PTVV}) on $\Spec\, A$ to the moduli space of homotopy (shifted) Poisson algebra structures on $A$. The first makes use of a more general description of the Poisson operad and of its cofibrant models, while the second is more computational and involves an explicit resolution of the Poisson operad.
\end{abstract}

\section*{Introduction}

In classical Poisson geometry, one defines a Poisson structure on a smooth manifold to be a Poisson bracket on the algebra of global functions, which is just a Lie bracket compatible with the product of functions. This notion (which is of algebraic nature) has a more geometric version. The geometric analog of skew-symmetric biderivations are bivector fields, and quite expectedly one can define a Poisson structure to be a bivector field satisfying some additional property. The equivalence of the two definitions of Poisson structure is a well-known fact in classical algebraic or differential geometry. It should be noted that the geometric approach becomes obligatory in the case of complex manifolds.

Recently, in their paper \cite{PTVV} Pantev, To\"en, Vaqui\'e and Vezzosi introduced the notion of symplectic and Poisson structures in the context of derived algebraic geometry.
Informally speaking, derived algebraic geometry is the study of spaces whose local models are \emph{derived commutative algebras}, that is to say simplicial commutative algebras. If we suppose to be working over a base field $k$ of characteristic zero, the local models can also be taken to be non-positively graded commutative dg-algebras.
See \cite{To} for a recent survey, or \cite{HAG1}, \cite{HAG2}, \cite{Lu} for a complete treatment of the subject.

In attempting to extend Poisson structures to derived algebraic geometry, there are thus two natural approaches: either via bivector fields or via Poisson brackets on the algebra of functions. We will show that these two approaches indeed agree for a huge and geometrically meaningful class of derived stacks. To do so, we will need to use techniques very different from the arguments in the non-derived setting.
To be a bit more specific, \cite{PTVV} use the bivector approach to define a $n$-Poisson structure on a (nice enough) derived algebraic stack. We refer to Section 1 for the precise definitions of the objects appearing below.
\begin{defi*}[ \cite{PTVV}, \cite{To} ] 
Let $X$ be a (nice enough\footnote{explicitly, $X$ has to be a derived Artin stack locally of finite presentation over $k$. This means in particular that its cotangent complex is perfect.}) derived algebraic stack, and let $n \in \Z$. The space of $n$-\emph{shifted Poisson structures on }$X$ is the simplicial set
$$\mathrm{Pois}(X,n) := \mathrm{Map}_{\mathsf{dgLie^{gr}}}(k[-1](2), \mathrm{Pol}(X,n)[n+1]) $$
where $k[-1](2)$ is concentrated in degree 1, pure of weight 2, and has the trivial bracket. The graded complex $\mathrm{Pol}(X,n)$ is the complex of $n$-\emph{shifted polyvector fields}.  
\end{defi*}

The purpose of this paper is to show that, at least for a nice enough affine derived stack $\Spec\, A$ (where $A$ is a derived commutative algebra), the equivalence between Poisson bivectors and Poisson brackets remains true. Our result is further evidence that for nice derived stacks the definition in \cite{PTVV} is the correct derived generalization of Poisson geometry. As we are working in an inherently homotopical context, Poisson brackets have to be given up to homotopy: these are basically $P_{n,\infty}$-structures on $A$ whose (weakly) commutative product is (equivalent to) the one given on $A$.

With this goal in mind, after having fixed our notational conventions in Section 1, we study in Section 2 the relation between the categories of dg-operads and of graded dg-Lie algebras. In particular, we would like to be able to describe the moduli space of Poisson brackets on a given commutative algebra via a mapping space in the category of graded dg-Lie algebras.
This is accomplished in greater generality in Theorem \ref{thm1}.

In Section 3, we apply the results of the previous section to derived algebraic geometry, and we eventually obtain the following result
\begin{teo*}
Let $A$ be a commutative dg algebra concentrated in degrees $(-\infty,m]$, with $m\geq0$, and let $\mathrm{End}_A$ be the (linear) endomorphism operad of the dg-module $A$. Let $X=\mathrm{Spec}\, A$ be the associated derived stack, and let $P^h_{n+1}(A)$ be the homotopy fiber of the morphism of simplicial sets
$$ \text{\emph{Map}}_{\text{\emph{dgOp}}}(\text{P}_{n+1}, \text{\emph{End}}_A) \longrightarrow  \text{\emph{Map}}_{\text{\emph{dgOp}}}(\text{\emph{Comm}}, \text{\emph{End}}_A) $$
taken at the point $\mu_A$ corresponding to the given (strict) multiplication in $A$.

Then there is a natural map in the homotopy category of simplicial sets
$$\mathrm{Pois}(X,n) \longrightarrow P_{n+1}^h(A) \ .$$

Moreover, this is an isomorphism if $\L_X$ is a perfect complex.
\end{teo*}
This is exactly the result we were looking for, since the simplicial set $P_{n+1}^h(A)$ is the natural moduli space of weak Poisson brackets on $A$.

Finally in Section 4 we give an alternative proof of this theorem, which is more computational and uses both an explicit resolution of the strict Poisson operad and the classical concrete definition of $L_{\infty}$-algebra.

The results in this paper are part of an ongoing project aimed at defining and studying higher quantizations of moduli spaces equipped with shifted Poisson structures. This has been recently achieved in the paper \cite{CPTVV}, where the authors give a more general definition of shifted Poisson structure using the results of this paper. For further details on the general project and its goals, we refer to the introductions of \cite{PTVV} and of \cite{CPTVV}, or to the survey \cite{To}.  \\

\noindent \textbf{Acknowledgments.} I am very grateful to Mauro Porta and Gabriele Vezzosi for many interesting and stimulating conversations that helped me during the writing of this paper. I also thank Damien Calaque, Marco Robalo, Nick Rozenblyum, Bertrand To\"en and Bruno Vallette for their useful remarks and suggestions, and the two anonymous referees for their specific and detailed comments. The present paper is part of my PhD thesis at Universit\'e Paris Diderot and Universit\`a di Firenze, under the supervision of G. Vezzosi and G. Ginot. The problem was suggested by my advisors, and also raised independently by N. Rozenblyum and J. Lurie.

\section{Notations}
\begin{itemize}
\renewcommand{\labelitemi}{$\bullet$}
\item $k$ is the base field, which is of characteristic 0.
\item $\mathsf{cdga^{\leq 0}}$ denotes the category of (strictly) commutative differential graded algebras, concentrated in non-positive degrees. We adopt the cohomological point of view, and the differential increases the degree by 1. The category $\mathsf{cdga}^{\leq 0}$ has the usual model structure for which weak equivalences are quasi-isomorphisms, and fibrations are surjections in negative degrees. 
\item $\mathsf{C(k)}$ denotes the category of unbounded cochain complexes over $k$. Its objects will be called also dg-modules. It has the usual model structure for which weak equivalences are the quasi-isomorphisms and fibrations are surjections. It is also a symmetric monoidal model category for the standard tensor product $\otimes_k$.
\item We will use the term \emph{symmetric sequence} to indicate a collection of dg-modules $\{V(m)\}_{m \in \N} $ such that every $V(m)$ has an action of the symmetric group $S_m$ on it. Explicitly, $V(m)$ is a differential graded $S_m$-module, meaning that for every $p \in Z$ the degree $p$ component $V(m)^p$ is an $S_m$-module, and that the differential is a map of $S_m$-modules. Equivalently, one can say that $V(m)$ is a differential graded $k[S_m]$-module, where $k[S_m]$ is the group algebra of $S_m$.
In the literature objects of this kind are sometimes called $\mathbb S$-modules, $\Sigma_*$-objects or also just collections in $C(k)$ (see for example \cite{BM} or Chapter 5 in \cite{LV}).
If $V$ is a symmetric sequence and $f \in V(m)$, we will denote by $f^{\s}$ the image of $f$ under the action of a permutation $\s \in S_m$. We will say that $f$ is \emph{symmetric} if $f^{\s} = f$ for every $\s \in S_m$. Similarly, we will say that $f$ is \emph{anti-symmetric} if $f^{\s}=(-1)^{\s} f$ for every $\s$, where $(-1)^{\s}$ denotes the sign of $\s$.
We will use the notation $V^{\mathbb S}$ for the symmetric sequence of invariants (i.e. of symmetric elements): explicitly, $V^{\mathbb S}(m) = V(m)^{S_m}$. We will allow ourselves to switch quite freely from the point of view of symmetric sequences to the one of graded dg-modules with an action of $S_m$ on the weight $m$ component.

Any symmetric sequence $V$ can be naturally seen as a functor from $C(k)$ to itself, sending a dg module $M$ to $\bigoplus (V(n) \otimes_{S_n} M^{\otimes n})$, where the $S_n$-action on $M^{\otimes n}$ is the natural one. Given two symmetric sequences $V$ and $W$, one can thus consider them as functors and take their composition; it can be shown that this composition comes from a symmetric sequence, denoted $V \circ W$.
\item $\mathsf{dgOp}$ is the category of monochromatic (i.e. uncolored) operads in the symmetric monoidal category $C(k)$ (dg-operads for short). It carries a model structure with componentwise quasi-isomorphisms as weak equivalences and componentwise surjections as fibrations (see \cite{Hi}). In particular, every dg-operad is fibrant. If $\mathcal P$ is a dg-operad, we denote by $\mathcal P_{\infty}$ a cofibrant replacement; then $\mathcal P_{\infty}$-algebras are up-to-homotopy $\mathcal P$-algebras. The operads of commutative algebras, of Lie algebras and of Poisson $n$-algebras will be denoted with Comm, Lie and $P_n$ respectively. Our convention is that a Poisson $n$-algebra has a Lie bracket of degree $1-n$; with this definition, the cohomology of a $E_n$-algebra is a $P_n$-algebra.
Notice however that there are other conventions in the literature: for example in \cite{CFL} the authors define a Poisson $n$-algebra to have Lie bracket of degree $-n$.
\item $\mathsf{dgLie^{gr}}$ is the category whose objects are graded dg-Lie algebras, that is to say graded dg-modules $L$ together with an antisymmetric binary operation $[\cdot , \cdot ] : L \otimes L \to L$ satisfying the (graded) Jacobi identity. The additional (i.e. the non-cohomological one) grading will be called \emph{weight}. The bracket must be of cohomological degree $0$ and of weight $-1$. Notice thus that these are not algebras for the trivial graded version of the Lie operad, since we are asking for the bracket to have weight $-1$.

The fact that the bracket has weight $-1$ is purely conventional: one can of course obtain the same results using brackets of weight 0. The seemingly strange choice is motivated by the observation that for an affine derived stack $\Spec\, A$, the natural bracket on the (shifted) polyvectors fields $\Sym_A (\T_A[-n])$ has weight $-1$. This is the same convention used for example in \cite{PTVV}.
\item Given a dg-module $V$, one defines its suspension $V[1]$ to be the cochain complex $V \otimes k[1]$, where $k[1]$ is the complex who is $k$ in degree $-1$ and $0$ elsewhere. If we do the same on operads, we should be a bit more careful. In fact, given an operad $\O$, the symmetric sequence $\O'(m)=\O(m)[1]$ does not inherit an operad structure. Instead, one defines the suspension of $\O$ to be the symmetric sequence whose terms are $s\O (m) = \O(m) [1-m] $, together with the natural operadic structure on it. A little more abstractly, $s\O$ is just $\O \otimes_H \mathrm{End}_{k[1]}$, where $\otimes_H$ denotes the Hadamard tensor product of operads (see \cite{LV}, Section 5.3.3). Note that the arity $p$ component of $\mathrm{End}_{k[1]}$ is $k[1-p]$; as a $S_p$-module, it is just the sign representation.
This operadic suspension is an auto equivalence of the category $\mathsf{dgOp}$, its inverse being a desuspension functor denoted $\O \mapsto s^{-1}\O$, and which sends $\O$ to $\O \otimes_H \mathrm{End}_{k[-1]}$.

\end{itemize}
\section{Operads and graded Lie algebras }

\subsection{The operad $\Lie$ and some generalizations}
In this section we study the dg-operad $\mathrm{Lie}$ and its cofibrant resolutions. Namely, we describe what it means to have a map from any of these dg-operads to another dg-operad $\O$.

We start by recalling how we can obtain a graded dg-Lie algebra $\Li(\O)$ in a natural way starting with a dg-operad $\O$ (see \cite{KM}, section 1.7). These are classical results in operad theory, and they play a very important role in the remainder of the paper.
\begin{prop}Let $\O$ be a dg-operad. Then the graded dg-module $\Li(\O) = \bigoplus_n \O(n)$ has a natural structure of  a graded dg-Lie algebra, where the Lie bracket is induced by the following pre-Lie product
$$f \star g = \sum_{i=1}^p \sum_{\s \in S_{p,q}^i} (f \circ_i g)^{\s} $$
where $f$ and $g$ are of weight (i.e. arity) $p$ and $q$ respectively, and where $S_{p,q}^i$ is the set of permutations of $p+q-1$ elements such that
$$ \s^{-1}(1)<	\s^{-1}(2)<\dots <\s^{-1}(i)<\s^{-1}(i+q) < \dots < \s^{-1}(p+q-1)   $$
and
$$ \s^{-1}(i)<\s^{-1}(i+1)<\dots < \s^{-1}(i+q-1) \ .$$
\end{prop}
Recall that one obtains a Lie bracket starting from a pre-Lie structure in a natural way: in our case, $[f,g]=f\star g - (-1)^{|f||g|}g\star f$. One has of course to check that the $\star$ operation defines a pre-Lie product (and therefore a Lie bracket): this is done by direct computation, showing that the so called \emph{associator} $f \star(g\star h) - (f \star g) \star h$ is (graded) symmetric on $g$ and $h$ (see \cite{LV}, Section 5.4.6).

The Lie bracket defined above has a first nice property: the following lemma is a straightforward consequence of the definition of the pre-Lie product.
\begin{lem}
Let $\O$ be a dg-operad, and let $f, g \in \Li(\O)$ be two symmetric elements. Then their bracket in $\Li(\O)$ remains symmetric.
\end{lem}
In particular, $\Li(\O)$ has a sub-Lie algebra of symmetric elements $\Li(\O)^{\mathbb S}$.

Our first goal is to use the construction of $\Li(\O)$ to find an alternative description to the set $\Hom_{\mathsf{dgOp}}(\mathrm{Lie},\O)$.

As an operad, $\mathrm{Lie}$ admits a very nice presentation : it is generated by a binary operation of degree 0 which is antisymmetric and satisfies the Jacobi identity. More specifically, if $l \in \mathrm{Lie}(2)_0$ is the generator, it has to satisfy $l \circ_1 l + (l \circ_1 l)^{(123)} + (l \circ_1 l)^{(132)} = 0$.

Thus we can safely say that
$$\Hom_{\mathsf{dgOp}}(\mathrm{Lie},\O) = \{ x \in \O(2)_0 | x^{(12)}=-x \ \text{and} \ x \circ_1 x + (x \circ_1 x)^{(123)} + (x \circ_1 x)^{(132)} = 0\}.$$

In Section 1, we defined the operadic suspension, which is an auto-equivalence of the category of dg-operads. The operad $s\mathrm{Lie}$ has one generator in arity 2 of degree 1, which is now symmetric; the Jacobi relation still holds in the same form, since it only involves even permutations. Note that algebras for this operad are just dg-Lie algebras whose bracket is of degree 1, or equivalently dg-modules $V$ with a dg-Lie algebra structure on $V[-1]$.
The operadic suspension being an equivalence, we have in particular $\Hom_{\mathsf{dgOp}}(\mathrm{Lie},\O) \cong \Hom_{\mathsf{dgOp}}(s\mathrm{Lie},s\O)$. Maps from the operad $s\mathrm{Lie}$ have a nice description in terms of maps of graded dg-Lie algebras.
\begin{prop} Let $\O$ be a dg-operad. Then we have
$$\Hom_{\mathsf{dgOp}}(s\mathrm{Lie},\O) \cong \Hom_{\mathsf{dgLie^{gr}}}(k[-1](2) ,\Li(\O)^{\mathbb S})$$
where $k[-1](2)$ is the graded dg-Lie algebra which has just $k$ in degree 1 and weight 2, with zero bracket, while $\Li$ is the functor $\mathsf{dgOp} \to \mathsf{dgLie^{gr}}$ defined at the beginning of this section.
\end{prop}
\begin{proof} It follows from the explicit presentation of $s\Lie$ given before that 
$$\Hom_{\mathsf{dgOp}}(s\Lie,\O) = \{ x \in \O(2)_1 | x^{(12)}=x \ \text{and} \ x \circ_1 x + (x \circ_1 x)^{(123)} + (x \circ_1 x)^{(132)} = 0\}$$
so that in order to prove the lemma we are led to show that the Jacobi relation is equivalent to the condition $[x,x]=0$ in $\Li(\O)$. This is done by direct calculation, since for any symmetric $x \in \O(2)$ we have
$$\begin{array}{ccl}
x \star x & = & x \circ_1 x + (x \circ_1 x)^{(23)} + (x \circ_2 x) \\
& = & x \circ_1 x + (x \circ_1 x)^{(123)} + (x \circ_1 x)^{(132)}
\end{array}$$
where we just use the general identities that describe the relationship between partial composition and the action of the symmetric groups. More specifically, take $f \in \O(p)$ and $g \in \O(q)$. Then for every $\s \in S_q$ one has
$$f \circ_i g^{\s} = (f \circ_i g)^{\s'} $$
where $\s' \in S_{p+q-1}$ acts as $\s$ on the block $\{i, i+1, \dots , i+q-1 \}$ and as the identity elsewhere. Moreover, for every $\tau \in S_p$, one has
$$f^{\t} \circ_i g = (f \circ_{\t(i)} g)^{\t'} $$
where $\t' \in S_{p+q-1}$ acts as the identity on the block $\{i, i+1, \dots , i+q-1 \}$ with values in $\{ \t(i), \t(i)+1, \dots , \t(i)+q-1 \}$ and as $\t$ elsewhere (sending $\{1,\dots ,p+q-1 \} \setminus \{ i, \dots , i+q-1 \}$ to $\{1,\dots ,p+q-1 \} \setminus \{ \t(i), \dots , \t(i)+q-1 \}$).

The lemma now follows from the observation that for an element $x \in \Li(\P)$ of degree 1, one has $[x,x] = 2(x \star x)$.
\end{proof}
One immediately has the following consequence.
\begin{cor}
For any dg-operad $\O$, we have
$$\Hom_{\mathsf{dgOp}}(\Lie, \O) \cong \Hom_{\mathsf{dgOp}}(s\Lie, s\O) \cong \Hom_{\mathsf{dgLie^{gr}}}(k[-1](2), \Li (s \O)^{\mathbb S}) \ . $$
\end{cor}

Next we try to find a result analogous to the last proposition for a cofibrant resolution $\widetilde{s\Lie}_Q$ of the dg-operad $s\Lie$. Our strategy is as follows: given a nice replacement $Q(k[-1](2))$ of $k[-1](2)$ as a graded dg Lie algebra, we find a ``lift'' of this replacement to a cofibrant approximation $\widetilde{s\Lie}_Q$ of the dg operad $s\Lie$. This construction will actually produce a functor from semi-free graded Lie algebras to semi-free operads, but we will not need this functoriality.

Suppose we have a semi-free resolution $Q(k[-1](2))$ of $k[-1](2)$ as a graded dg-Lie algebra. This means that if we forget the differential $Q(k[-1](2))$ is a free graded Lie algebra, say with generators $\{ p_i \}_{i \in I}$, homogeneous of degree $d_i$ and of weight $w_i$. Then there are of course relations $\{r_j\}_{j \in J}$ that can specify the value of $d(p_i)$, where $d$ is the differential. We can now use this resolution to build a dg-operad $\widetilde{s\Lie}_Q$.

Concretely, for every $i \in I$, take a symmetric generator $\widetilde{p_i}$ of arity $w_i$ and of degree $d_i$. As for relations, we take the same relations $r_j$ defining $Q(k[-1](2))$; this means that whenever such a relation $r_j$ contains a bracket $[p_{i_1},p_{i_2}]$, we reinterpret it as the bracket (introduced at the beginning of this section) of elements of an operad $[\widetilde{p}_{i_1}, \widetilde{p}_{i_2}]$, thus getting a relation $\widetilde{r_j}$ for the generators $\widetilde{p}_i$. Let us denote $\widetilde{s\Lie}_Q$ the semi-free operad having all the $\widetilde{p_i}$ as generators and all the $\widetilde{r}_j$ as relations.
The definition of the operad $\widetilde{s\Lie}_Q$ allows us to describe quite naturally the set of morphism $\Hom_{\mathsf{dgOp}}(\widetilde{s\Lie}_Q, \O)$ for an arbitrary operad $\O$.
One can in fact prove the following result.

\begin{prop}\label{resolutions}
Let $Q(k[-1](2))$ be a semi-free resolution of the graded dg-Lie algebra $k[-1](2)$, and let $\widetilde{s\Lie}_Q$ be the operad defined above, which has the same generators and relations of $Q(k[-1](2))$, and such that all generators are symmetric. Then for every dg-operad $\O$ we have
$$\Hom_{\mathsf{dgOp}}(\widetilde{s\Lie}_Q, \O) \cong  \Hom_{\mathsf{dgLie^{gr}}}(Q(k[-1](2)), \Li (\O)^{\mathbb S}) \ . $$
\end{prop}
\begin{proof} This follows from the definition of $\widetilde{s\Lie}_Q$ in terms of generators and relations. 
Just like what we said before for $s\Lie$, morphisms form $\widetilde{s\Lie}_Q$ are completely determined by the images of the generators $\widetilde{p_i}$, provided that they satisfy the relations defining $\widetilde{s\Lie}_Q$.
Every relation can be expressed inside $\Li(\widetilde{s\Lie}_Q)$, since they only specify the differentials of the $\widetilde{p_i}$ in terms of their brackets. And by definition these relations of course coincide with those of $Q(k[-1](2))$, giving the desired result.
\end{proof}

In particular, we observe that $\widetilde{s\Lie}_Q$ is a cofibrant approximation of $s\Lie$: the weak equivalence $\widetilde{s\Lie}_Q \to s\Lie$ is induced by the weak equivalence $Q(k[-1](2)) \to k[-1](2)$. The fact that it is cofibrant follows from the definition of cofibrations in the model category of dg-operads, given in \cite{Hi}.

The operad $\widetilde{s\Lie}_Q$ is therefore weakly equivalent to any cofibrant replacement of $s\Lie$. This is just a consequence of the existence of the dotted arrow in the following commutative diagram
\begin{center}
\begin{tikzpicture}[description/.style={fill=white,inner sep=5pt}] 
\matrix (m) [matrix of math nodes, row sep=3em, 
column sep=3.5em, text height=2ex, text depth=0.25ex] 
{\emptyset & (s\Lie)_{\infty} \\
\widetilde{s\Lie}_Q  & s\Lie \\};
\path[->,font=\scriptsize] 
(m-1-1) edge node[auto] {$ $} (m-1-2)
(m-1-1) edge node[auto] {$ $} (m-2-1)
(m-1-2) edge node[auto] {$ $} (m-2-2)
(m-2-1) edge node[auto] {$ $} (m-2-2);
\draw[dotted,->] (m-2-1) -- (m-1-2);
%edge node[auto] {$ \varphi_{\alpha(l)} $} (m-2-2) 
%(m-2-2) edge node[auto] {$ A $} (m-2-3);
%(m-2-1) edge node[auto,swap] {$ r_A \otimes 1_B $} (m-3-1); 
\end{tikzpicture} 
\end{center}

Since the operadic suspension preserves weak equivalences and fibrations, the map $s(\Lie_{\infty}) \to s\Lie$ is a trivial fibration, where $\Lie_{\infty}$ is the standard minimal model of the operad $\Lie$, studied for example by Markl in \cite{Mar}.
In particular, it follows that $\widetilde{s\Lie}_Q$ is weakly equivalent to $s(\Lie_{\infty})$. Once again this is just a consequence of the existence of a model category structure on $\mathsf{dgOp}$, which assures that the dotted arrow in the following diagram
\begin{center}
\begin{tikzpicture}[description/.style={fill=white,inner sep=5pt}] 
\matrix (m) [matrix of math nodes, row sep=3em, 
column sep=3.5em, text height=2ex, text depth=0.25ex] 
{\emptyset & s(\Lie_{\infty}) \\
\widetilde{s\Lie}_Q  & s\Lie \\};
\path[->,font=\scriptsize] 
(m-1-1) edge node[auto] {$ $} (m-1-2)
(m-1-1) edge node[auto] {$ $} (m-2-1)
(m-1-2) edge node[auto] {$ $} (m-2-2)
(m-2-1) edge node[auto] {$ $} (m-2-2);
\draw[dotted,->] (m-2-1) -- (m-1-2);
%edge node[auto] {$ \varphi_{\alpha(l)} $} (m-2-2) 
%(m-2-2) edge node[auto] {$ A $} (m-2-3);
%(m-2-1) edge node[auto,swap] {$ r_A \otimes 1_B $} (m-3-1); 
\end{tikzpicture} 
\end{center}
exists, and that it is a weak equivalence.

Note that this does not imply that 
$$\Hom_{\mathsf{dgOp}}(\Lie_{\infty}, \O) \cong \Hom_{\mathsf{dgOp}}(s\Lie_{\infty}, s\O) \cong \Hom_{\mathsf{dgLie^{gr}}}(Q(k[-1](2)), \Li (s \O)^{\mathbb S})$$
since $\widetilde{s\Lie}_Q$ and $s\Lie_{\infty} $ are not isomorphic in general.

\subsection{Derivations and multi-derivations}

Our goal now is to define shifted Poisson brackets on a commutative algebra, and hence we need to understand derivations of a commutative algebra in an operadic way. 

Recall that for a commutative dg-algebra $A$ we have a standard notion of \emph{multi-derivation}. Namely one says that a linear map $\phi : A^{\otimes p} \to A$ is a multi-derivation if for every $i=1,\dots p$ and for every choice of $a_1,\dots,\hat{a_i}, \dots a_p \in A$ the induced linear map
$$\begin{array}{ccl} A& \longrightarrow& A \\
x &\longmapsto &\phi(a_1,\dots, a_{i-1},x,a_{i+1},\dots a_p)
\end{array}$$
is a (graded) derivation of $A$.
More generally, for every operadic morphism $\mu  : \Comm \to \O$, we can say what it means for any element of $\O$ to be a \emph{derivation} with respect to $\mu$. Notice that the map $\mu$ is completely determined by the image in $\O(2)$ of the generator of the operad $\Comm$; in order to simplify the notation, we will also use the letter $\mu$ to denote the image of the generator.
\begin{defi}\label{multider}  Let $\O$ be a dg-operad, and let $\mu :  \Comm \to \O$ be a morphism of dg-operads. Suppose $f \in \O(p)$ is an element of $\O$ of arity $p \in \N$. We say that $f$ is a \emph{$p$-derivation with respect to $\mu$} if we have
$$f \circ_i \mu = (\mu \circ_1 f)^{(p+1 \ p\ \dots\ i+2 \ i+1)} + (\mu \circ_2 f)^{(1 \ 2 \ \dots \ i-1 \ i)} $$
for every $i=1, \dots, p$. The symmetric sub-sequence of $\O$ formed by $p$-derivations will be denoted by $\mathcal{MD}(\O,\mu)$, and its elements will just be called \emph{multi-derivations} with respect to $\mu$. If the morphism $\mu$ is clear from the context, we will just write $\MD(\O)$.
\end{defi}

The definition is coherent with the classical case of derivations of an algebra: if $ \O $ is the endomorphism operad of a dg-module $V$ and $\mu$ is an actual commutative product on $V$ (so that $(V, \mu)$ is just a commutative dg-algebra), then multi-derivations in our sense are exactly multi-derivations in the standard sense.

Let us remark that one could give a definition analogous to Definition \ref{multider} that works for every element $\mu \in \O(2)$, of any degree, and without making any assumption on the symmetry of $\mu$. For example, a derivation with respect to such a $\mu$ is just an element $f \in \O(1)$ such that
$$f \circ \mu = (-1)^{ |\mu| |f| } \mu \circ_1 f + (-1)^{ |\mu| |f| } \mu \circ_2 f \ .$$
In order to generalize this to multi-derivations, one should keep track of the signs.
\begin{defi}\label{multider2} Let $\O$ be a dg-operad, and let $\mu \in \O(2)$. An element $f \in \O(p)$ is called a $p$-\emph{derivation with respect to} $\mu$ if for every $i=1,2,\dots,p$ we have
$$f \circ_i \mu = (-1)^{|\mu| |f|}(\mu \circ_1 f)^{(i+1\ p+1 \ p\ \dots\ i+2)} + (-1)^{|\mu| |f|}(\mu \circ_2 f)^{(1\ 2\ \dots\ i-1 \ i )}$$
\end{defi}

The dg-module of derivations of an algebra $A$ is known to be a dg-Lie algebra in a natural way: the (graded) commutator of two derivations is in fact still a derivation. Derivations thus form a sub-Lie algebra of $\Hom_{\mathsf{dgMod}}(A,A)$. More can be said, since actually the graded module of multi-derivations of $A$ is a graded sub-Lie algebra of $\Li(\mathrm{End}_A)$. The following lemma tells us that the same remains true in the world of operads.
\begin{prop} Let $\O$ be a dg-operad, and let $\mu \in  \O(2)$ be a binary operation. The (graded module associated to the) symmetric sequence of multi-derivations with respect to $\mu$ of Definition \ref{multider2} is closed under the Lie bracket of $\Li(\O)$.
\end{prop}
\begin{proof}
This follows from a straightforward computation: let us give the main ideas without going into all the details. We will suppose that $\mu$ is of even degree in order to avoid keeping track of too many signs. The proof for $\mu$ of odd degree is exactly the same, with additional signs of course.

Let $f \in \O(p)$ and $g \in \O(q)$ be two multi-derivations with respect to $\mu$. We have
$$[f,g] \circ_i \mu = (f \star g) \circ_i \mu + (-1)^{|f||g|}(g \star f) \circ_i \mu 	\ ,$$
and we would like to show that this is equal to 
\begin{align*}
(\mu \circ_1 [f,g])^{(p+q \ \dots \ i+1)} + (\mu \circ_2 [f,g])^{(1\ \dots \ i)}
& = (\mu \circ_1 (f\star g))^{(p+q \ \dots \ i+1)} +  \\
& \  + (-1)^{|f||g|}(\mu \circ_1 (g\star f))^{(p+q \ \dots \ i+1)} + \\ & \ +(\mu \circ_2 (f\star g))^{(1 \ \dots \ i)}\\
&\  + (-1)^{|f||g|}(\mu \circ_2 (g\star f))^{(1 \ \dots \ i)}
\end{align*}
Notice that just as with composition of vector fields, $f \star g$ and $g \star f$ have no hope of being multi-derivations themselves, and one really has to develop the sums in order to prove the result.
Using the relations between partial compositions and the action of the symmetric groups, we may write
$$(f \star g) \circ_i \mu  = \sum_{j=1}^p \sum_{ \s \in S_{p,q}^j} (f \circ_j g)^{\s} \circ_{i} \mu 
=  \sum_{j=1}^p \sum_{ \s \in S_{p,q}^j} ((f \circ_j g) \circ_{\s(i)} \mu)^{\s'} \ .$$
We now observe that if $\s(i) \notin \{j,j+1, \dots , j+q-1\}$, then we can just use the fact that $f$ is a derivation, and we are done. A similar reasoning applies to $(g\star f)\circ_i \mu$.
When $\s(i) \in \{j,j+1, \dots , j+q-1\}$, it gets a bit more complicated. With some care, we can write down what it is left to prove, that is
\begin{equation*}
\setlength{\jot}{-13pt}
\begin{split}
\sum_{j=1}^p \sum_{\substack{ \s \in S_{p,q}^j \\ j\leq \s(i) < j+q }}  & ((\mu\circ_2 f ) \circ_1 g )^{\varphi \cdot ( j+q \ j+q-1 \ \dots \ \s(i)+1 )\cdot \s'} = \\
&\quad \hspace{22pt} =  (-1)^{|f||g|}\sum_{k=1}^q \sum_{\substack{ \t \in S_{q,p}^k \\ k\leq \t(i) < k+p }} ((\mu \circ_1 g ) \circ_{q+1} f )^{\psi \cdot (k\ k+1\ \dots \ \t(i))\cdot \t'}
\end{split}
\end{equation*}
where $\varphi \in S_{p+q}$ is the permutation that exchanges the blocks $\{1,\dots,j-1\}$ and $\{j,\dots, j+q-1\}$, and $\psi \in S_{p+q}$ is the permutation that exchanges the blocks $\{k+1,\dots, k+p\}$ and $\{k+p+1,\dots,p+q\}$. This last equation is true by direct verification: both sides are equal to the sum of all possible ``products'' of the form $\mu \circ ( f, g)$.

%\begin{align*}
%(f \star g) \circ_i \mu & = \sum_{j=1}^p \sum_{ \s \in S_{p,q}^j} (f \circ_j g)^{\s} \circ_{\s(i)} \mu \\
% & =  \sum_{j=1}^p \sum_{ \s \in S_{p,q}^j} ((f \circ_j g) \circ_i \mu)^{\s'} \\
% & = \sum_{j=1}^p \sum_{\substack{ \s \in S_{p,q}^j \\ j\leq \s(i) < j+q }} (f \circ_j (g \circ_{\s(i)-j+1} %\mu))^{\s'} + \sum_{j=1}^p \sum_{\substack{ \s \in S_{p,q}^j \\ \s(i) < j }} ((f \circ_{\s(i)} \mu ) \circ_{j+1} %g )^{\s'} + \\
%& \ \ + \sum_{j=1}^p \sum_{\substack{ \s \in S_{p,q}^j \\ j\leq \s(i) \geq j+q }} ((f \circ_{\s(i)-q+1} \mu ) %\circ_{j} g )^{\s'}
%\end{align*}

\end{proof}

\subsection{The operad $\widetilde{P}_{n,Q}$}

Recall (see Section 8.6 of \cite{LV} and references therein) that given two operads $\P$ and $\Q$, if we choose a morphism of symmetric sequences $\Lambda : \Q \circ \P \to \P \circ \Q$ (satisfying a series of axioms), then we can put an operad structure on the composite of the underlying symmetric sequences $\P \circ \Q$. The idea is that in order to define a composition $(\P \circ \Q) \circ (\P \circ \Q) \to \P \circ \Q$, we can use the morphism $\Lambda$ followed by the given compositions $\P \circ \P \to \P$ and $\Q \circ \Q \to \Q$, coming from the operad structures on $\P$ and $\Q$.
Informally speaking, $\Lambda$ specifies how the operations encoded by the operad $\P$ interact with those encoded by $\Q$. Such a $\Lambda$ is called a \emph{distributive law}, because of the motivating example of the relation between the sum and the multiplication in a ring. When $\P$ and $\Q$ have a nice presentation in terms of generators and relations, we only need a \emph{rewriting rule} for the generators (we refer again to Section 8.6 of \cite{LV} for more details).

We now let $Q(k[-1](2))$ be again a semi-free resolution of the graded dg-Lie algebra $k[-1](2)$: as before, we can associate to it an operad $\widetilde{s\Lie}_Q$, which is quasi-isomorphic to $s\Lie_{\infty}$. 
The operad introduced in the following definition will play a central role in the remainder of the paper.
\begin{defi}
Let $Q(k[-1](2))$ be again a semi-free resolution of the graded dg-Lie algebra $k[-1](2)$, and let as before $\widetilde{s\Lie}_Q$ be the operad of Proposition \ref{resolutions}. We define the operad $\widetilde{P}_{n,Q}$ to be the operad obtained by means of a rewriting rule out of $s^{-n}\widetilde{s\Lie}_Q$ and $\mathrm{Comm}$, imposing the condition that every generator of $s^{-n}\widetilde{s\mathrm{Lie}}_Q$ is a multi-derivation with respect to the generator of $\mathrm{Comm}$.
Explicitly, if we denote by $s^{-n}\widetilde{p}_i$ the generators of $s^{-n}\widetilde{s\mathrm{Lie}}_Q$ and by $\mu$ the generator of $\Comm$, the rewriting rule sends $s^{-n}\widetilde{p}_i \circ_k \mu$ to $(\mu \circ_1 f)^{(k+1\ p+1 \ p\ \dots\ i+2)} + (\mu \circ_2 f)^{(1\ 2\ \dots\ k-1 \ k )}$.
% More precisely, the map $s^{-n}\widetilde{s\Lie}_Q \circ_{(1)} \mathrm{Comm} \to \Comm \circ_{(1)} s^{-n}\widetilde{s\Lie}_Q$
\end{defi}
It is clear from the definition that $\widetilde{P}_{n,Q}$-algebras are commutative dg-algebras $A$ with a compatible $\widetilde{s\Lie}_Q$-structure on $A[n]$, where the compatibility is given by the condition that the operations defining the $\widetilde{s\Lie}_Q$-structure must be multi-derivations of the commutative dg algebra $A$.
This operad is obviously weakly equivalent to 
%For our purposes, we will be only interested in multi-derivations which are invariant for the action of the symmetric group: these will be called \emph{symmetric multi-derivations}. From the previous discussion it is obvious that they are closed under the Lie bracket of $\Li(\O)$; we will denote this sub-Lie algebra by $\Li_{Sym}^{md}(\O)$, leaving the dependance from $\mu$ implicit.
the dg-operad obtained in a similar way out of $\Comm$ and $s\Lie_{\infty}$ (recall that with $\Lie_{\infty}$ we mean the minimal model of the operad $\Lie$). Let us call $\widehat{P}_n$ this latter operad. More specifically, $\widehat{P}_n$-algebras are commutative dg-algebras $A$ together with a $\Lie_{\infty}$-structure on $A[n-1]$. The two structures are compatible, meaning that the multi-brackets defining the (shifted) $\Lie_{\infty}$-structure are multi-derivations on the algebra $A$.

\begin{rem}
The operad $\widehat{P}_n$ (actually a non-shifted version of it) has appeared for instance in \cite{CF}, where the authors called its algebras \emph{flat $P_{\infty}$-algebras}. However, as Cattaneo and Felder correctly remarked in their paper, their notation is a bit misleading, because $\widehat{P}_n$ is not a cofibrant replacement of the operad $P_n$: in particular the product encoded in $\widehat{P}_n$ is strictly commutative. One could see $\widehat{P}_n$ as an operad standing between the original $P_n$ and its minimal model $P_{n,\infty}$. For this reason, algebras for $\widehat{P}_n$ will be called \emph{semi-strict $P_n$-algebras}. For an explicit definition of $P_{n,\infty}$-algebras in term of generators and relations in the case $n=2$ (corresponding to homotopy Gerstenhaber algebras), one can look at \cite{Gi}.
\end{rem}

By construction, the operad $\widetilde{P}_{n,Q}$ has a natural map from the commutative dg-operad $\Comm$. If $\O$ is any dg-operad, we now describe the fiber of the induced morphism
$\Hom_{\mathsf{dgOp}}(\widetilde{P}_{n,Q},\O) \to \Hom_{\mathsf{dgOp}}(\Comm, \O) $ at a point $\mu$. In particular, if we take $\O$ to be the endomorphism operad of a dg-module $V$, we are studying the possible ways in which a given commutative structure on $V$ can be extended to a $\widetilde{P}_{n,Q}$-structure. From the very definition of $\widetilde{P}_{n,Q}$, it is clear that what we are missing is a shifted $\widetilde{s\Lie}_{Q}$-structure made out of multi-derivations. Luckily the preceding results give us exactly a way to compute those structures.
\begin{prop}\label{set-fib}
 Let $\O$ be a dg-operad, and let $\mu : \Comm \to \O$ be a map of operads. The fiber at $\mu$ of the map
$$\mathrm{Hom}_{\mathsf{dgOp}}(\widetilde{P}_{n,Q},\O) \to \mathrm{Hom}_{\mathsf{dgOp}}(\Comm, \O) $$
is the set $\Hom_{\mathsf{dgLie^{gr}}}(Q(k[-1](2)), \Li(s^{n} \MD(\O))^{\mathbb S})$ .

\end{prop}
\begin{proof} By definition of the operad $\widetilde{P}_{n,Q}$, the fiber we are trying to compute is a subset of $\Hom_{\mathsf{dgOp}}(\widetilde{s\Lie}_{Q}, s^{n}\O)$: in fact, it is composed of morphisms $s^{-n}\widetilde{s\Lie}_{Q} \to \O$. The condition they must satisfy is that the image of the generators must be multi-derivations with respect to $\mu$. It follows that our fiber is the subset of maps $\widetilde{s\Lie}_{Q} \to s^n\O$ which send generators to suspensions of multi-derivations.
Using Proposition \ref{resolutions}, we thus get that the fiber is exactly $\Hom_{\mathsf{dgLie^{gr}}}(Q(k[-1](2)), \Li(s^{n} \MD(\O))^{\mathbb S})$. 
Notice that it may seem that we are being a bit inaccurate here, as it is not entirely obvious that the (operadic) suspensions of elements of the sub-Lie algebra $\MD(\O)$ are still a sub-Lie algebra of $\Li(s^n\O)$. This is nonetheless true, and it follows from the observation that elements in $s^{n}\MD(\O)$ are exactly multi-derivations with respect to the $n$-suspension of the commutative product $\mu$. To see this, take $f$ a multi-derivation of $\O$ of arity $p$. We want to show that image under the operadic suspension of $f$ is a multi-derivation of $s\O$ with respect to the suspension of $\mu$. This would easily imply our claim, and therefore the theorem.

Recall that the component of arity $p$ of $s\O$ is $\O(p) \otimes k[1-p]$, where $k[1-p]$ is the signature representation of $S_p$ put in degree $p-1$. We denote the generator of $k[1-p]$ by $x_{p-1}$, so that $|x_{p-1}|=p-1$. By definition of the compositions in $s\O$, we have
\begin{align*}
(f\otimes x_{p-1}) \circ_i (\mu \otimes x_{1}) & =  (f \circ_i \mu ) \otimes (x_{p-1}\circ_i x_1) \\
& =  (\mu \circ_1 f)^{(i+1\ p+1 \dots i+2)}\otimes (x_{p-1}\circ_i x_1 )+  \\
& \ \ \ + (\mu \circ_1 f)^{(i\ p+1 \dots i+1)}\otimes (x_{p-1}\circ_i x_1) \\
& =  (-1)^{p-i}((\mu \circ_1 f) \otimes (x_{p-1} \circ_ix_1))^{(i+1\ p+1 \dots i+2)}+ \\
& \ \ \ + (-1)^{p+1-i}((\mu \circ_1 f) \otimes (x_{p-1}\circ_i x_1))^{(i\ p+1 \dots i+1)}
\end{align*}
Now observe that
\begin{align*}(\mu \circ_1 f)\otimes (x_{p-1} \circ_i x_1) & = (-1)^{i-1}(\mu \circ_1 f)\otimes (x_{1}\circ_1 x_{p-1}) \\
& =  (-1)^{i-1}(-1)^{|f|}(\mu \otimes x_1) \circ_1 (f \otimes x_{p-1})
\end{align*}
so that we have
\begin{align*}
(f\otimes x_{p-1}) \circ_i (\mu \otimes x_{1})  =  & \ (-1)^{p-1}(-1)^{|f|}((\mu \otimes x_1) \circ_1 (f \otimes x_{p-1}))^{(i+1\ p+1 \dots i+2)}\ - \\
 & \ - (-1)^{p-1}(-1)^{|f|}((\mu \otimes x_1) \circ_1 (f \otimes x_{p-1}))^{(i\ p+1 \dots i+1)}
\end{align*}
which tells us exactly that $f \otimes x_{p-1}$ is a multi-derivation with respect to the binary operation $\mu \otimes x_1$.
\end{proof}

\subsection{The moduli space of $\widetilde{P}_{n,Q}$-structures}

We are now ready to prove our first main result. Given two dg-operads $\P$ and $\Q$, one can form a simplicial space of morphisms from $\P$ to $\Q$, which we will denote by $\underline{\Hom}_{\mathsf{dgOp}}(\P,\Q)$. Namely, we can construct a simplicial resolution $\Q_{\bullet}$ of $\Q$ and consider the simplicial set whose n-simplices are $\Hom(\P, \Q_n)$ and whose face and degeneracy maps are the ones induced by the simplicial structure of $\Q_{\bullet}$. Notice that this is not the derived mapping space between $\P$ and $\Q$ in the model category of dg-operads, since we are not replacing $\P$ with a cofibrant model. If the operad $\P$ is cofibrant, then $\underline{\Hom}_{\mathsf{dgOp}}(\P,\Q)$ is isomorphic to the mapping space $\mathrm{Map}_{\mathsf{dgOp}}(\P,\Q)$ in the homotopy category of simplicial sets.

The simplicial set $\underline{\Hom}_{\mathsf{dgOp}}(\P,\Q)$ has a nice interpretation if we put $\Q = \mathrm{End}_V$, where $V$ is a dg-module. In this case, $\underline{\Hom}_{\mathsf{dgOp}}(\P,\mathrm{End}_V)$ can be thought of as a sort of moduli space of $\P$-algebra structures on $V$.

We can ask whether Proposition \ref{set-fib} remains true at the level of simplicial sets. First of all, we remark that the question makes sense: for every operad $\O$, we have a naturally induced map
$\underline{\mathrm{Hom}}_{\mathsf{dgOp}}(\widetilde{P}_{n,Q}, \O) \to \underline{\mathrm{Hom}}_{\mathsf{dgOp}}(\Comm, \O)$ (induced by the natural morphism of operads $\Comm \to \widetilde{P}_{n,Q}$) that forgets the additional structure, and we could wonder if we can describe the fiber of a 0-simplex $\mu \in \underline{\mathrm{Hom}}_{\mathsf{dgOp}}(\Comm, \O)$ in terms of some simplicial set of morphisms in the category $\mathsf{dgLie^{gr}}$. The following theorem answers this question affirmatively.

\begin{teo}\label{thm1}
 Let $\O$ be a dg-operad, and let $\mu : \Comm \to \O$ be a map of operads. The (strict) fiber at $\mu$ of the morphism of simplicial sets
$$\underline{\mathrm{Hom}}_{\mathsf{dgOp}}(\widetilde{P}_{n,Q},\O) \to \underline{\mathrm{Hom}}_{\mathsf{dgOp}}(\Comm, \O) $$
is the simplicial set $\mathrm{\Hom}_{\mathsf{dgLie^{gr}}}(Q(k[-1](2)), \Li(s^{n}\MD(\O))^{\mathbb S} \otimes \Om_*)$, which is a right homotopy function complex from $k[-1](2)$ to $\Li(s^{n}\MD(\O))^{\mathbb S}$ in the model category of graded dg-Lie algebras.
\end{teo}

Before proving the theorem, we give explicit ways to compute simplicial resolutions and mapping spaces in both model categories $\mathsf{dgOp}$ and $\mathsf{dgLie^{gr}}$.

Let $L \in \mathsf{dgLie^{gr}}$. We can construct new graded dg-Lie algebras from $L$ by extension of scalars from $k$ to any $k$-dg-algebra. Let us define $\Om_n$ to be the dg-algebra of algebraic differential forms on $\Spec\left(\faktor{k[t_0, \dots, t_n]}{t_0+ \dots + t_n = 1} \right)$. As an algebra, we have
$$\Omn = k[t_0, \dots , t_n , dt_0 , \dots , dt_n ]/(1-\sum t_i , \sum dt_i)$$
where the generators $t_i$ have degree $0$ and the $dt_i$ have degree $1$. The algebras $\Omn$ define a simplicial object in the category of commutative dg-algebras in a natural way.
Then the simplicial graded dg-Lie algebra $L \otimes \Om_*$ is a simplicial resolution of $L$. Hence in $\mathsf{ dgLie^{gr}}$, the mapping space between two objects $L$ and $M$ has an explicit representative. Its $n$-simplices are
$$ \text{Map}_{\mathsf{dgLie}^{\mathsf{gr}}}(L,M) _n \ = \ \text{Hom}_{\mathsf{dgLie}^{\mathsf{gr}}}(Q(L),M\otimes_k \Om_n) \ ,$$
where Q is a cofibrant replacement of $L$.

Just as for graded dg-Lie algebras, given an operad $\O$ we can construct new operads by extension of scalars. 

\begin{prop}
For a dg-operad $\O$, the simplicial object $\O \otimes_k \Om_*$ (defined as above) gives a fibrant simplicial framing of the operad $\O$ (i.e. a fibrant replacement of $\O$ in the Reedy model category of simplicial objects in $\mathsf{dgOp}$).
\end{prop}

This follows directly from \cite{Fr}, Part II, Chapter 7 (in particular Theorem 7.3.5).

%\begin{proof} Let us recall that with simplicial framing of a dg-operad $\O$, we mean any simplicial object $\O_{\bullet}$ such that $\O_0 = \O$, and where:
%\begin{itemize}
%\item[(1)] for every $n \in \N$, the degeneracy map $\O_0 \to \O_n$ is a weak equivalence,
%\item[(2)] the collection of the "vertices" map $\displaystyle \O_n \to \prod_{i=0}^n \O_0$ defines a Reedy fibration in the category of simplicial objects in $\mathrm{dgOp}$. 
%\end{itemize}
%Condition (1) is automatic, since $\Om_0^{\bullet} \cong k$. Weak equivalences of dg-operads are exactly component-wise weak equivalences, as constructed in \cite{Hi}, \cite{Hi2};
%\end{proof}

\begin{proof}[Proof of Theorem \ref{thm1}]
By construction, the $m$-simplices of the fiber are the $m$-simplices of the simplicial set $\underline{\Hom}_{\mathsf{dgOp}}(\widetilde{P}_{n,Q}, \O)$ that are sent to $\mu$, viewed as a degenerate $m$-simplex of $\underline{\Hom}_{\mathsf{dgOp}}(\Comm, \O)$. Therefore we can use Proposition \ref{set-fib} in order to compute them: they are the fiber of the function
$$\mathrm{Hom}_{\mathsf{dgOp}}(\widetilde{P}_{n,Q},\O \otimes \Om_m) \to \mathrm{Hom}_{\mathsf{dgOp}}(\Comm, \O \otimes \Om_m) $$
taken at the point $\mu$. Notice that we are being a bit sloppy in order to keep notation as simple as possible, as we are identifying $\mu : \Comm \to \O$ with the composition $\Comm \to \O \to \O \otimes \Omn$.
So Proposition \ref{set-fib} tells us that the $m$-simplices of the fiber are
$\Hom_{\mathsf{dgLie^{gr}}}(Q(k[-1](2)), \Li(s^{n} \MD(\O \otimes \Om_m))^{\mathbb S})$.

Observe now that multi-derivations of $\O \otimes \Omn$ are just multi-derivations of $\O$ with respect to $\mu$, considered over the dg-algebra $\Omn$. Concretely, this means
$\MD(\O \otimes \Omn) = \MD(\O) \otimes \Omn$ as graded dg-Lie algebras. Moreover, the operadic suspension commutes with extension of scalars, as does taking invariants. It follows that the graded dg-Lie of $n$-simplices is
$$\Hom_{\mathsf{dgLie^{gr}}}(Q(k[-1](2)), \Li(s^{n} \MD(\O))^{\mathbb S} \otimes \Om_m) = \mathrm{Map}_{\mathsf{dgLie^{gr}}}(k[-1](2), \Li(s^{n}\MD(\O))^{\mathbb S})_m$$
where the $\mathrm{Map}$ on the right is computed by means of the right homotopy function complex described before. These isomorphisms organize in a natural way to give an isomorphism of simplicial set between the fiber at $\mu$ and a right homotopy function complex $\mathrm{Map}_{\mathsf{dgLie^{gr}}}(k[-1](2), \Li (s^n \MD (\O))^{\mathbb S})$, and this proves the theorem.
\end{proof}

\section{Applications to derived algebraic geometry}

Let again $Q(k[-1](2))$ be a semi-free resolution of the dg-Lie algebra $k[-1](2)$. In this section we apply Theorem \ref{thm1} to the context of derived Poisson geometry. In particular, we will show in Theorem \ref{thm2} that a $n$-Poisson structure in the sense of \cite{PTVV} on a derived stack of the form $\Spec\, A$ (with $A$ concentrated in degree $(-\infty, m]$, with $m\geq0$) gives rise to a $\widetilde{P}_{n+1,Q}$-structure on $A$.

Recall from \cite{PTVV} that for a derived Artin stack $X$ which is locally of finite presentation, the space $\mathrm{Pois}(X,n)$ of $n$-Poisson structures on $X$ is by definition the mapping space $\mathrm{Map}_{\mathsf{dgLie^{gr}}}(k[-1](2),\mathrm{Pol}(X,n)[n+1])$, where $\mathrm{Pol}(X,n)$ is the graded Poisson dg-algebra of $n$-shifted polyvectors, that is to say
$$\mathrm{Pol}(X,n)=\R \G(X, \Sym_{\O_X} \T_X[-n-1])\ .$$
If $X=\Spec\, A$ is affine, $\mathrm{Pol}(X,n)$ becomes just $\Sym_A(\T_A[-n-1])$ with the usual Schouten-Nijenhuis bracket.

Recall also from \cite{Qu} that the category of bounded above cochain complexes have a natural model structure, taking as weak equivalences the quasi-isomorphisms and as fibrations the degree-wise surjections. This structure induces in the standard way (via the free-forgetful adjunction) a model structure on bounded above commutative dg algebras.

\begin{teo}\label{thm2}
Let $A$ be a cofibrant object in the model category of commutative dg algebras that are bounded above. Suppose that $A$, viewed as a derived stack, admits a n-shifted Poisson structure in the sense of \cite{PTVV}. Then $A$ has a structure of an $\widetilde{P}_{n+1,Q}$-algebra, whose commutative product coincide with the given multiplication in $A$.
More precisely, let $\mu_A$ be the multiplication in $A$, and let $\widetilde{P}_{n+1,Q}(A)$ be the fiber of the map of simplicial sets
$$ \underline{\mathrm{Hom}}_{\mathsf{dgOp}}(\widetilde{P}_{n+1,Q}, \text{\emph{End}}_A) \longrightarrow  \underline{\mathrm{Hom}}_{\mathsf{dgOp}}(\text{\emph{Comm}}, \text{\emph{End}}_A) $$
at the point $\mu_A$. 

We have a natural map of simplicial set
$$\mathrm{Pois}(\mathrm{Spec}\,A,n) \longrightarrow \widetilde{P}_{n+1,Q}(A)$$
%where $\mathrm{Pois}(\mathrm{Spec}(A),n) = \mathrm{Map}_{\mathrm{dgLie}^{\mathrm{gr}}}(k[-1](2),\text{\emph{Sym}}_A(\T_A[-n-1])[n+1])$ is the simplicial set of $n$-Poisson structures on $\Spec(A)$ introduced in \cite{PTVV}.

Moreover, this map is a weak equivalence if the cotangent complex $\L_A$ is perfect.
\end{teo}
\begin{proof}
The simplicial set $\widetilde{P}_{n+1,Q}(A)$ has an equivalent description given by Theorem \ref{thm1}, namely we can rewrite it as $\mathrm{Hom}_{\mathsf{dgLie^{gr}}}(Q(k[-1](2)), \Li(s^{n+1}\MD(A))^{\mathbb S} \otimes \Om_*)$, where $\MD(A)$ is the Lie algebra of multi-derivations of the operad $\mathrm{End}_A$, with respect to the natural multiplication $\mu_A : \Comm \to \mathrm{End}_A$ (see Definition \ref{multider}); that is to say, the classically defined multi-derivations of the algebra $A$.  By functoriality, in order to prove the theorem it will suffice to build up a map of graded dg-Lie algebras
$$ \text{Sym}_A(\T_A[-n-1])[n+1] \longrightarrow \Li(s^{n+1}\MD(A))^{\mathbb S} 	\ .$$

To construct this morphism, notice that since $A$ is cofibrant, $\L_A$ is just the standard module of K\"ahler differentials, and multi-derivations of $A$ of arity $p$ are by definition the $A$-module $\Hom_{A}(\L_A^{\otimes p}, A)$. Hence the weight $p$ component of the graded dg-Lie algebra $\Li(s^{n+1}\MD(A))^{\mathbb S}$ is precisely given by the symmetric elements inside $\Hom_A(\L_A^{\otimes p},A) \otimes k[1-p]^{\otimes (n+1)}$, where $k[1-p]$ is the signature representation of $S_p$ concentrated in degree $p-1$. As an $S_p$-module, $k[1-p]^{\otimes n}$ can be either a trivial or a signature representation, depending on the parity of $n$. Concretely, we have
$$k[1-p]^{\otimes (n+1)} = \begin{cases}
\text{the trivial representation of } S_p  & \text{ if } n \text{ is odd }\\
\text{the signature representation of } S_p & \text{ if } n \text{ is even }
\end{cases}$$
where the $S_p$-modules are always concentrated in degree $(n+1)(p-1)$. It follows that as a dg-module, the weight $p$ part of $\Li(s^{n+1}\MD(A))^{\mathbb S}$ is isomorphic to $\Hom_A(\Sym_A^p\L_A,A)[(n+1)(1-p)]$ if $n$ is odd, and to $\Hom_A(\Lambda_A^p\L_A,A)[(n+1)(1-p)]$ if $n$ is even.

On the other hand, the weight $p$ component of $\text{Sym}_A(\T_A[-n-1])[n+1]$ is just $\text{Sym}^p_A(\T_A[-n-1])[n+1]$, and we have a natural map of $k$-dg-modules (actually of $A$-dg-modules)
$$\text{Sym}^p_A(\T_A[-n-1])[n+1] \longrightarrow \Hom_{A}(\Sym^p_A (\L_A[n+1]), A)[n+1] $$
induced by the fact that $\T_A$ is by definition the dual of $\L_A$. Notice that this map is not an equivalence in general: it becomes an equivalence however if we suppose that $\L_A$ is perfect.
Observe next that we have
$$\Sym^p_A (\L_A[n+1])= \begin{cases} \Sym^p_A(\L_A )[n(p-1)]  & \text{ if $n$ is odd }\\
 \Lambda^p_A(\L_A )[n(p-1)] & \text{ if $n$ is even }
\end{cases}$$
so that for every $n$, $\Hom_{A}(\Sym^p_A (\L_A[n+1]), A)[n+1] $ is isomorphic as a dg-module to  the weight $p$ component of $\Li(s^{n+1}\MD(A))^{\mathbb S}$.

Putting all this together, we do get a map of graded dg-modules
$$ \text{Sym}_A(\T_A[-n-1])[n+1] \longrightarrow \Li(s^{n+1}\MD(A))^{\mathbb S}  \ .$$
The point is to check that this map is compatible with the two Lie brackets: on the left hand side, we have the Schouten bracket, induced by the natural Lie structure on $\T_A$, while on the right hand side we have the bracket of the Lie algebra associated to the operad $s^{n+1}\mathrm{End}_A= \mathrm{End}_{A[n+1]}$.

This can be done by direct calculation, since both brackets have a known explicit expression. One has just to check that the signs coincide.

More abstractly, we can also observe that there is an adjunction
$$\left\{ \begin{array}{c} A\text{-dg-modules $X$ with a }  \\k\text{-linear dg-Lie structure on $X[m]$} \end{array} \right\} \leftrightarrows \left\{ \begin{array}{c} \text{commutative } A\text{-dg-algebras $X$ with a} \\ \text{compatible }  k\text{-linear dg-Lie structure on $X[m]$} \end{array} \right\}$$
where on the right hand side, compatible means that if we forget the $A$-action we are left with a $P_{m+1}$-algebra. Alternatively, these are just $P_{m+1}$-algebras in $C(k)$ whose underlying commutative algebra is actually an $A$-algebra, with no relation between the Poisson bracket and the $A$-action.

The adjunction is thus a ``lift'' of the usual free-forget adjunction between $A$-modules and $A$-algebras to the situation where the underlying $k$-modules have Lie structures.
The right adjoint is the forgetful functor, while the left adjoint sends $X$ to $\Sym_A(X)$. In particular this implies that if we were able to show that $\Li(s^{n+1}\MD(A))^{\mathbb S}[-n-1]$ has a compatible $A$-algebra structure, then the existence of a Lie algebra map
$$ \text{Sym}_A(\T_A[-n-1])[n+1] \longrightarrow \Li(s^{n+1}\MD(A))^{\mathbb S}  \ .$$
would follow from the existence of a morphism of Lie algebras (and of $A$-modules)
$$\T_A \longrightarrow \Li(s^{n+1}\MD(A))^{\mathbb S} \ .$$
But it follows from the definitions that the weight one component of $\Li(s^{n+1}\MD(A))^{\mathbb S}$ is precisely $\T_A$, and that the restriction of the bracket of $\Li(s^{n+1}\MD(A))^{\mathbb S}$ to $\T_A$ is the natural one (that is to say the graded commutator).

We are thus left to define an appropriate degree zero product on $\Li(s^{n+1}\MD(A))^{\mathbb S}[-n-1]$. It turns out that it is induced by the natural shuffle product on the multilinear morphisms from $A[n+1]$ to itself, which has the following explicit description. Denote by $\mu$ the multiplication of $A$; for $f \in \mathrm{End}_{A[n+1]}(p)$ and $g \in \mathrm{End}_{A[n+1]}(q)$, we pose
$$f \cdot g = \sum_{\s \in \mathrm{Sh}_{p,q}} (s^{n+1}\mu(f,g))^{\s}$$
where the sum is taken over all permutations $\s \in S_{p+q}$ such that $\s^{-1}(1)<\dots <\s^{-1}(p)$ and $\s^{-1}(p+1)<\dots<\s^{-1}(p+q)$. It easy to check that this defines a degree $m$ product, which becomes commutative if regarded on $\Li(\mathrm{End}_{A[n+1]})[-n-1]$. Moreover, if $f$ and $g$ are symmetric multi-derivations, then $f \cdot g$ is again a symmetric multi-derivation.
Finally, the graded Leibniz identity
$$[f,g\cdot h] = [f,g] \cdot h + (-1)^{|g|(|f|+n+1)}g\cdot[f,h]$$
for $f,g,h \in \Li(s^{n+1}\MD(A))^{\mathbb S}[-n-1]$ should be checked to be true. Notice that here the product $g\cdot h$ denotes the operation induced by the shuffle product defined above: this means that there are other signs involved, due to the so-called d\'ecalage isomorphism. The verification of the identity is a long but straightforward computation, and we omit the details.
To summarize, $\Li(s^{n+1}\MD(A))^{\mathbb S}[-n-1]$ is an $A$-algebra with a $k$-linear compatible Lie bracket of degree $-n-1$, and by the discussion above this proves the theorem.
\end{proof}

We can rephrase the results of Theorem \ref{thm2} in a different way:  we constructed a map of simplicial sets
$$ \text{Map}_{\mathsf{dgLie}^{\mathsf{gr}}}(k[-1](2),\text{Sym}_A(\T_A[-n-1])[n+1]) \xrightarrow{\ \phi\ } \underline{\text{Hom}}_{\mathsf{dgOp}}(\widetilde{P}_{n+1,Q}, \text{End}_A)$$
that fits in the following diagram
\begin{center}
\begin{tikzpicture}[description/.style={fill=white,inner sep=5pt}] 
\matrix (m) [matrix of math nodes, row sep=3em, 
column sep=3.5em, text height=1.5ex, text depth=0.25ex] 
{ \mathrm{Pois}(\Spec\, A,n) & \widetilde{P}_{n+1,Q}(A)  & \text{\underline{Hom}}_{\text{dgOp}}(\widetilde{P}_{n+1,Q} , \text{End}_A) \\
 & \text{\emph{pt}} & \text{\underline{Hom}}_{\text{dgOp}}(\text{Comm}, \text{End}_A) \\ };
\draw (-3.5,1.2) arc (150:30:3.7 and 1.2) [->];
\draw (-0.2,2.1) node {$\phi$};
\path[->,font=\scriptsize] 
(m-1-1) edge node[auto] {$ $} (m-1-2)
(m-1-2) edge node[auto] {$ $} (m-1-3)
edge node[auto] {$ $} (m-2-2)
(m-1-3) edge node[auto] {$ $} (m-2-3)
%edge node[auto] {$ \varphi_{\alpha(l)} $} (m-2-2) 
(m-2-2) edge node[auto] {$ \mu_A $} (m-2-3);
%(m-2-1) edge node[auto,swap] {$ r_A \otimes 1_B $} (m-3-1); 
\end{tikzpicture} 
\end{center}
where the square on the right is a pullback of simplicial sets.

Let us weaken a bit our results in order to express them in a more homotopical language. The following theorem is the main result of this text.
\begin{teo}\label{thm3}
Let $A$ be a commutative dg algebra concentrated in degree $(-\infty,m]$, and let $X= \mathrm{Spec}\, A$ be the derived stack associated to $A$. Let $P_{n+1}^h(A)$ be the homotopy fiber of the morphism of simplicial sets
$$ \text{\emph{Map}}_{\text{\emph{dgOp}}}(\text{P}_{n+1}, \text{\emph{End}}_A) \longrightarrow  \text{\emph{Map}}_{\text{\emph{dgOp}}}(\text{\emph{Comm}}, \text{\emph{End}}_A) $$
taken at the point $\mu_A$ corresponding to the given (strict) multiplication in $A$.

Then there is a natural map in the homotopy category of simplicial sets
$$\mathrm{Pois}(X,n) \longrightarrow P_{n+1}^h(A) \ .$$

Moreover, this is an isomorphism if $\L_X$ is a perfect complex.
\end{teo}
\begin{proof}

Notice that since we are only looking for a morphism in the homotopy category of simplicial sets, we can safely suppose that $A$ is cofibrant: in fact, the homotopy type of both $\mathrm{Pois}(X,n)$ and $P_{n+1}^h(A)$ does not change if we replace $A$ with another algebra quasi-isomorphic to it. 

As already mentioned towards the end of Section 2, the mapping space between two operads $\P$ and $\Q$ can be computed by taking a cofibrant replacement of the first one and a simplicial resolution of the second one. Let us denote by $C$ the cofibrant replacement functor in the model category of dg-operads. In particular, one has 
$$\mathrm{Map}_{\mathsf{dgOp}}(\P,\Q) \cong \underline{\Hom}_{\mathsf{dgOp}}(C(\P),\Q)\ .$$ Notice that we don't need to replace $\Q$ with a fibrant model, since all operads are fibrant.

In order to compute the homotopy fiber $P_{n+1}^h(A)$, one has thus to take cofibrant models for the operad $\Comm$ and $P_{n+1}$. For example, let us take the minimal model $\Comm_{\infty}$ of $\Comm$, and take $P_{n+1,\infty}$ to be the operad whose algebras are $\Comm_{\infty}$-algebras together with a $\widetilde{s\Lie}_Q$-structure on $A[n]$ made of homotopy derivations, in the sense of \cite{DL}, \cite{Do}. This just means that the generators of $\widetilde{s\Lie}_Q$ satisfy the Leibniz identity only up to homotopy.

These are clearly cofibrant models for $\Comm$ and $P_{n+1}$, and there is an obvious forgetful functor $\Comm_{\infty} \to P_{n+1,\infty}$, that is actually easily seen to be a cofibration in $\mathsf{dgOp}$ using the characterization of cofibrations given in \cite{Hi}. This means that the induced morphism 
$$ \underline{\Hom}_{\mathsf{dgOp}}(P_{n+1,\infty}, \mathrm{End}_A)
\longrightarrow  \underline{\Hom}_{\mathsf{dgOp}}(\Comm_{\infty}, \mathrm{End}_A) $$  
is a fibration between fibrant simplicial sets, and therefore its \emph{strict} fiber is weakly equivalent to its homotopy fiber, which in turn is a model for $P_{n+1}^h(A)$, the homotopy fiber of
$$ \mathrm{Map}_{\mathsf{dgOp}}(P_{n+1}, \mathrm{End}_A) \longrightarrow  \mathrm{Map}_{\mathsf{dgOp}}(\Comm, \mathrm{End}_A) \ \ .$$

Let us now consider the following diagram of simplicial sets:

\begin{center}
\begin{tikzpicture}[description/.style={fill=white,inner sep=5pt}] 
\matrix (m) [matrix of math nodes, row sep=3em, 
column sep=3.5em, text height=1.5ex, text depth=0.25ex] 
{ \text{\underline{Hom}}_{\mathsf{dgOp}}(\overline{P}_{n+1,Q}, \text{End}_A) & \text{\underline{Hom}}_{\mathsf{dgOp}}(P_{n+1,\infty}, \text{End}_A) \\
\text{\underline{Hom}}_{\text{dgOp}}(\text{Comm}, \text{End}_A) & \text{\underline{Hom}}_{\text{dgOp}}(\text{Comm}_{\infty}, \text{End}_A) \\ };
\path[->,font=\scriptsize] 
(m-1-1) edge node[auto] {$ $} (m-1-2)
(m-1-1) edge node[auto] {$ $} (m-2-1)
(m-1-2) edge node[auto] {$ $} (m-2-2)
(m-2-1) edge node[auto] {$ $} (m-2-2);
%edge node[auto] {$ \varphi_{\alpha(l)} $} (m-2-2) 
%(m-2-2) edge node[auto] {$ A $} (m-2-3);
%(m-2-1) edge node[auto,swap] {$ r_A \otimes 1_B $} (m-3-1); 
\end{tikzpicture} 
\end{center}
where $\overline{P}_{n+1,Q}$ is the operad whose algebras are strictly commutative algebras together with a $\widetilde{s\Lie}_Q$-structure on $A[n]$ made of homotopy derivations. By definition, this is a pullback diagram of simplicial sets, so that the strict fiber of the map on the right (taken at the point $\mu$) is equivalent to the strict fiber of the map on the left (still taken at $\mu$; this makes sense since $\mu$ factors through $\underline{\Hom}_{\mathsf{dgOp}}(\Comm, \mathrm{End}_A)$).

But now the strict fiber of the map
$$ \underline{\Hom}_{\mathsf{dgOp}}(\overline{P}_{n+1,Q}, \mathrm{End}_A)
\longrightarrow  \underline{\Hom}_{\mathsf{dgOp}}(\Comm, \mathrm{End}_A) $$ 
is the space of $\widetilde{s\Lie}_Q$-structures on $A[n]$ made of homotopy derivations. Our next goal is now to describe this space, that can actually be quite complicated for a general $A$.

There is a naturally defined dg module $\mathrm{Der}^h(A)$ of homotopy derivation of $A$, which can be used to compute the Hochshild cohomology of the algebra $A$. Namely, instead of resolving $A$ and then computing strict derivations, one can leave $A$ unresolved and compute homotopy derivations (see \cite{Do}, section 3). In particular, this shows that for a cofibrant algebra $A$ one has a quasi-isomorphism $\mathrm{Der}(A) \cong \mathrm{Der}^h(A)$, where $\mathrm{Der}(A)$ is the standard complex of strict derivations of $A$. Let us remark that this result should not come as a surprise, since both $\mathrm{Der}(A)$ and $\mathrm{Der}^h(A)$ are in this case sensible candidates for the tangent complex of the algebra $A$, and one should expect no ambiguity in the definition of such a geometrically meaningful object.

In particular this tells us that for $A$ cofibrant, the space of $\widetilde{s\Lie}_Q$-structures on $A[n]$ made of homotopy derivations is weakly equivalent to the space of $\widetilde{s\Lie}_Q$-structures on $A[n]$ made of strict derivations; but this last space is by definition $\widetilde{P}_{n+1,Q}(A)$.
Now Theorem \ref{thm2} gives us a map of simplicial sets from $\mathrm{Pois}(A;n)$ to $\widetilde{P}_{n+1,Q}(A)$, which corresponds to a map in the homotopy category of simplicial sets from $\mathrm{Pois}(A;n)$ to $P_{n+1}^h(A)$.

\

We conclude by observing that the last statement of the theorem is a direct consequence of the analogous statement in Theorem \ref{thm2}.

\end{proof}

\section{Another proof of the main result}

In this last section we will give a more explicit description of our results: we take a particular resolution of the graded dg-Lie algebra $k[-1](2)$ and study the induced resolution of the $\Lie$ operad. We check that its algebras are just $\Lie_{\infty}$-algebras in the standard sense, see for example \cite{HS}. These concrete computations also give an alternative proof of Theorem \ref{thm3}.

The graded dgLie algebra $k[-1](2)$ has a cofibrant resolution $L_0$ given by the free Lie algebra generated by elements $p_i$ for $i=2,3,\dots$, such that $p_i$ sits in weight $i$ and in cohomological degree 1 ; the differential in $L_0$ is defined as to satisfy 
$$dp_n = -\frac 12 \sum_{i+j=n+1} [p_i,p_j] $$
Notice that in particular that we have $dp_2=0$. The map $L_0 \to k[-1](2)$ sends $p_2$ to the generator of $k[-1](2)$ and the other $p_i$ to zero.

By definition, the space of $n$-shifted Poisson structures on $A$ is 
$$\mathrm{Map}_{\mathsf{dgLie^{gr}}}(k[-1](2),\text{Sym}_A(\T_A[-n-1])[n+1])$$
and we can use the explicit resolution $L_0$ to compute its n-simplices: these are just elements in 
$$\text{Hom}_{\mathsf{dgLie}^{\mathsf{gr}}}(L_0,\text{Sym}_A(\T_A[-n-1])[n+1] \otimes \Omega_n)\ .$$
In particular, the points of the space of $n$-shifted structures on $A$ can be identified with
$$\text{Hom}_{\mathsf{dgLie}^{\mathsf{gr}}}(L_0,\text{Sym}_A(\T_A[-n-1])[n+1])\ .$$ 
If $A$ is cofibrant, then the dg module of derivations of $A$ is a model for $\T_A$. In this case the same argument used in the proof of Theorem \ref{thm2} proves that $\text{Sym}_A(\T_A[-n-1])[n+1]$ maps into the dg Lie algebra of symmetric (shifted) multi-derivations $\mathcal L(s^{n+1}\MD(A))^{\mathbb S}$, which in turn sits inside $\mathcal L(s^{n+1}\mathrm{End}_A)^{\mathbb S} \cong \mathcal L (\mathrm{End}_{A[n+1]})^{\mathbb S}$, the dg Lie of all symmetric multilinear maps of $A[n+1]$.

Putting all together, we get a map from the points of $\mathrm{Pois}(A,n)$ to
$$\text{Hom}_{\mathsf{dgLie}^{\mathsf{gr}}}(L_0,\bigoplus_{i \in \N} \Hom_k(\Sym^i_k(A[n+1]),A[n+1]))\ .$$

So at the level of the vertices, a $n$-Poisson structure on $A$ gives a sequence of symmetric multilinear maps $q_i$ (the images of the $p_i$) on $A[n+1]$, such that every $q_i$ is an $i$-linear map of degree $1$.

One of the possible definitions (see for example \cite{Man}) of a $L_{\infty}$-structures is the following.
\begin{defi}
If $V$ is a graded vector space, an $L_{\infty}$-structure on $V$ is a sequence of symmetric maps of degree $1$
$$l_n : \Sym^n  V[1] 
\rightarrow V[1] \ , \ \ 
n>0 $$
such that for every $n>0$ we have
$$\sum_{i+j=n+1}[l_i,l_j] = 0 \ , $$
where the bracket is the Lie bracket we defined before on $\displaystyle \bigoplus_{i \in \N} \Hom_k(\Sym^i_k(V[1]),V[1])$.
\end{defi}

So if we want to prove that (still at the level of the vertices) an $n$-Poisson structure
gives us an $L_{\infty}$-structure on $A[n]$, we could try to find such $l_n$ on $A[n+1]$. Natural candidates are the $q_i$ that come directly from the shifted Poisson structure; these are given for $i>1$. 
Notice that our brackets satisfy the graded antisymmetry relation $[x,y]=-(-1)^{|x||y|}[y,x]$ ; in particular, this relation does not involve the weights of $x$ and $y$. In our case $|p_i|=|q_i|=1$, and so it follows $[p_i,p_j]=[q_i,q_j]=[p_j,p_i]=[q_j,q_i]$. 
Let us take $q_1=d$, the differential of $A[n+1]$. We should now verify that the symmetric maps $q_i$ satisfy 
$$\sum_{i+j=n+1}[q_i,q_j] = 0 \ . $$

The other observation we need to make is that for every multilinear map
$f \in \Hom_k(\Sym^i_k(A[n+1]),A[n+1]))$,  we have $[q_1,f]=[f,q_1]=d(f)$, where $d$ here is the differential of multilinear maps on $A[n+1]$.

So using these facts we have
$$\sum_{i+j=n+1}[q_i,q_j] = 2d(q_n) + \sum_{\substack{i+j=n+1 \\  i,j>1}} [q_i,q_j] = 0 \ ,$$
which is what we wanted.
To summarize, an $n$-Poisson structure induces an $L_{\infty}$-structure on $A[n]$.

Now we need to show that the induced $L_{\infty}$-structure on $A[n]$ is compatible with the algebra structure on $A$, that is to say that $A$ becomes a semi-strict $P_{n+1}$-algebra. But the $q_i$ we constructed in the previous step are (by definition) derivations of the given commutative product on $A$ ; this gives $A$ precisely the structure of a semi-strict $P_{n+1}$-algebra.

The upshot of this discussion is the fact that we got a map
$$ \text{Hom}_{\mathsf{dgLie}^{\mathsf{gr}}}(L_0,\text{Sym}_A(\T_A[-n-1])[n+1]) \longrightarrow \text{Hom}_{\mathsf{dgOp}}(\widehat{P}_{n+1}, \text{End}_A)$$
for which the image is contained in the $\widehat{P}_{n+1}$-structures whose commutative product is the one given on $A$. Equivalently, we get a function from $ \text{Hom}_{\mathsf{dgLie^{gr}}}(L_0,\text{Sym}_A(\T_A[-n-1])[n+1]) $ to the (non-homotopical) fiber product of the following diagram of sets
\begin{center}
\begin{tikzpicture}[description/.style={fill=white,inner sep=5pt}] 
\matrix (m) [matrix of math nodes, row sep=3em, 
column sep=3.5em, text height=1.5ex, text depth=0.25ex] 
{  & \text{Hom}_{\text{dgOp}}(\widehat{P}_{n+1}, \text{End}_A) \\
\text{\emph{pt}} & \text{Hom}_{\text{dgOp}}(\text{Comm}, \text{End}_A) \\ };
\path[->,font=\scriptsize] 
(m-1-2) edge node[auto] {$ $} (m-2-2)
%edge node[auto] {$ \varphi_{\alpha(l)} $} (m-2-2) 
(m-2-1) edge node[auto] {$ \mu_A$} (m-2-2);
%(m-2-1) edge node[auto,swap] {$ r_A \otimes 1_B $} (m-3-1); 
\end{tikzpicture} 
\end{center}
where $\mu_A$ denotes the given commutative product of $A$. From here one can proceed in the exact same way as done towards the end of Section 2: namely, we can use Theorem \ref{thm1} (and the explicit descriptions of the simplicial framings in dgOp and $\mathsf{dgLie^{gr}}$) in order to prove that we have a map of simplicial sets from $\text{\underline{Hom}}_{\mathsf{dgLie}^{\mathsf{gr}}}(L_0,\text{Sym}_A(\T_A[-n-1])[n+1])$ to the (strict) fiber of the natural map $ \text{\underline{Hom}}_{\mathsf{dgOp}}(\widehat{P}_{n+1}, \text{End}_A) \to \text{\underline{Hom}}_{\mathsf{dgOp}}(\mathrm{Comm}, \text{End}_A)$, taken at $\mu_A$. 

Now the same arguments used at the end of Section 3 allow to obtain a map in the homotopy category of simplicial sets
$$\mathrm{Pois}(X,n) \longrightarrow P_{n+1}^h(A)$$
giving a more concrete proof of Theorem \ref{thm3}.

\end{document}